\documentclass[12pt]{amsart} %

\usepackage{float}

\usepackage{epsfig}
\usepackage{tikz}
\usepackage{array,amsmath, enumerate,  url, psfrag}
\usepackage{amssymb, amsaddr, fullpage}
\usepackage{graphicx,subfigure}
\usepackage{color}

\newtheorem{theorem}{Theorem}[section]
\newtheorem{lemma}[theorem]{Lemma}

\newtheorem{corollary}[theorem]{Corollary}

\newtheorem{thm}{Theorem}[section]
\newtheorem{lem}[thm]{Lemma}

\newtheorem{definition}[thm]{Definition}

\title{Planar graphs without cycles of lengths $4$ and $5$ and close triangles are DP-3-colorable}

\author{Yuxue Yin$^{1}$ \and   Gexin Yu$^{1,2}$}

\address{
$^{1}$\small Department of Mathematics, Central China Normal University, Wuhan, Hubei, 430079 China.\\
$^2$\small Department of Mathematics, The College of William and Mary, Williamsburg, VA, 23185, USA.
}

\thanks{The research of the last author was supported in part by the Natural Science Foundation of China (11728102) and the NSA grant H98230-16-1-0316.}

\email{gyu@wm.edu}

\begin{document}
\maketitle

\begin{abstract}
Montassier, Raspaud, and Wang (2006) asked to find the smallest positive integers $d_0$ and  $d_1$ such that planar graphs without $\{4,5\}$-cycles and $d^{\Delta}\ge d_0$ are $3$-choosable and planar graphs without $\{4,5,6\}$-cycles and $d^{\Delta}\ge d_1$ are $3$-choosable, where $d^{\Delta}$ is the smallest distance between triangles.  They showed that $2\le d_0\le 4$ and $d_1\le 3$. In this paper, we show that the following planar graphs are DP-3-colorable: (1) planar graphs without $\{4,5\}$-cycles and $d^{\Delta}\ge 3$ are DP-$3$-colorable, and (2) planar graphs without $\{4,5,6\}$-cycles and $d^{\Delta}\ge 2$ are DP-$3$-colorable. DP-coloring is a generalization of list-coloring, thus as a corollary, $d_0\le 3$ and $d_1\le 2$. We actually prove stronger statements that each pre-coloring on some cycles can be extended to the whole graph.
\end{abstract}

\section{Introduction}
Coloring of planar graphs has a long history.   The famous Four Color Theorem states that every planar graph is properly $4$-colorable, where a graph is properly $k$-colorable if there is a function $c$ that assigns an element $c(v) \in [k]:=\{1,2,\ldots, k\}$ to each $v \in V(G)$ so that adjacent vertices receive distinct colors.

Gr\"otzsch~\cite{G59} showed every planar graph without 3-cycles is 3-colorable. But it is NP-complete to decide whether a planar graph is $3$-colorable.  There were heavy research on sufficient conditions for a planar graph to be $3$-colorable.  Three typical conditions are the following:

\begin{itemize}
\item One is in the spirit of the Steinberg's conjecture (recently disproved) or Erd\H{o}s's problem that forbids cycles of certain lengths. Borodin, Glebov, Raspaud, and Salavatipour \cite{BGRS05} showed that planar graphs without $\{4,5,6,7\}$-cycles are $3$-colorable, and it remains open to know if one can allow $7$-cycle.

\item Havel \cite{H69} proposed to make $d^{\Delta}$ large enough, where $d^{\Delta}$ is the smallest distance between triangles. Dvo\^{r}\'{a}k, Kral, and Thomas\cite{DKT16} showed that $d^{\Delta}\ge 10^{100}$ suffices.

\item The Bordaux approach \cite{BR03} combines the two kinds of conditions. Borodin and Glebov \cite{BG10} showed that planar graphs without $5$-cycles and $d^{\Delta}\ge 2$ are $3$-colorable. It is conjectured \cite{BR03} that $d^{\Delta}\ge 1$ suffices.
\end{itemize}

Vizing~\cite{V76}, and independently Erd\H{o}s, Rubin, and Taylor~\cite{ERT79} introduced list coloring as a generalization of proper coloring. A \emph{list assignment} $L$ gives each vertex $v$ a list $L(v)$ of available colors. A graph $G$ is {\em $L$-colorable} if there is a proper coloring $c$ of $V(G)$ such that $c(v)\in L(v)$ for each $v\in V(G)$. A graph $G$ is {\em $k$-choosable} if $G$ is $L$-colorable for each $L$ with $|L(v)|\ge k$. Clearly, a proper $k$-coloring is an $L$-coloring when $L(v)=[k]$ for all $v\in V(G)$. 


While list coloring provides a powerful tool to study coloring problems, some important techniques used in coloring (for example, identification of vertices) are not feasible in list coloring. Therefore, it is often the case that a condition that suffices for coloring is not enough for the corresponding list-coloring. Thomassen~\cite{T94,T95} showed that every planar graph is $5$-choosable and every planar graph without $\{3,4\}$-cycles is 3-choosable, but  Voigt~\cite{V93, V95} gave non-$4$-choosable planar graphs and non-$3$-choosable triangle-free planar graphs.

Sometimes we do not know if a stronger condition would help.  For example, Borodin (\cite{B96}, 1996) conjectured that  planar graphs without cycles of lengths from $4$ to $8$ are $3$-choosable.

In the spirit of Bordeaux conditions, Montassier, Raspaud, and Wang~\cite{MRW06} gave the following conditions for a planar graph to be $3$-choosable:

\begin{theorem}[Montassier, Raspaud, and Wang \cite{MRW06}]\label{MRW}
A planar graph $G$ is $3$-choosable if
\begin{itemize}
\item $G$ contains no cycles of lengths $4$ and $5$ and $d^{\Delta}\ge 4$, or
\item $G$ contains no cycles of lengths from $4$ to $6$ and $d^{\Delta}\ge 3$.
\end{itemize}
There exist planar graphs without $4$-, $5$-cycles and $d^{\Delta}=1$ that are not $3$-choosable.
\end{theorem}

They asked for the optimal conditions on $d^{\Delta}$ for the same conclusions.

Very recently, Dvo\^{r}\'{a}k and Postle~\cite{DP17} introduced DP-coloring (under the name correspondence coloring), which helped them confirm the conjecture by Borodin mentioned above. DP-coloring is a generalization of list-coloring, but it allows identification of vertices in some situations.

\begin{definition}
Let $G$ be a simple graph with $n$ vertices, and $L$ be a list assignment of $V(G)$. For each vertex $u\in V(G)$, let $L_u=\{u\}\times L(u)$. For each edge $uv$ in $G$, let $M_{uv}$ be a matching (maybe empty) between the sets $L_u$ and $L_v$ and let $\mathcal{M}_L = \{ M_{uv} : uv \in E(G)\}$, called the \emph{matching assignment}. Let $G_L$ be the graph that satisfies the following conditions
\begin{itemize}
\item $V(G_L) = \cup_{u\in V(G)} L_u$.
\item for all $u \in V(G)$, the set $L_u$ forms a clique.
\item if $uv \in E(G)$, then the edges between $L_u$ and $L_v$ are those of $M_{uv}$
\item if $uv \notin E(G)$, then there are no edges between $L_u$ and $L_v$
\end{itemize}
If $G_L$ contains an independent set of size $n$, then $G$ has an {\em $\mathcal{M}_L$-coloring}. The graph $G$ is {\em DP-$k$-colorable} if, for each $k$-list assignment $L$ and each matching assignment $\mathcal{M}_L$ over $L$, it has an $\mathcal{M}_L$-coloring. The minimum $k$ such that $G$ is DP-$k$-colorable is the {\em DP-chromatic number} of $G$, denoted by $\chi_{DP}(G)$.
\end{definition}

As in list coloring, we refer to the elements of $L(v)$ as colors and call the element $i\in L(v)$ chosen in the independent set of an $\mathcal{M}_L$-coloring as the color of $v$.


We should note that DP-coloring and list coloring can be quite different. For example, Bernshteyn~\cite{B16} showed that the DP-chromatic number of every graph $G$ with average degree $d$ is $\Omega(d/\log d)$, while Alon \cite{A00} proved that $\chi_l(G)=\Omega(\log d)$ and the bound is sharp.


Much attention was drawn on this new coloring, see for example, \cite{B16,B17,BK17a,BK17b,BKP17,BKZ17,KO17a,KO17b,KY17, LLNSY18, LLRYY18b}.   We are interested in DP-coloring of planar graphs. Dvo\v{r}\'ak and Postle~\cite{DP17} noted that Thomassen's proofs~\cite{T94} for choosability can be used to show $\chi_{DP}(G)\le5$ if $G$ is a planar graph, and $\chi_{DP}(G)\le3$ if $G$ is a planar graph with no 3-cycles and 4-cycles. Some sufficient conditions were given in \cite{KO17a,KO17b,LLNSY18} for a planar graph to be DP-$4$-colorable.  Sufficient conditions for a planar graph to be DP-$3$-colorable are obtained in \cite{LLYY18} and \cite{LLRYY18b}. In particular,

\begin{theorem}\label{DP3}(\cite{LLYY18, LLRYY18b})
A planar graph is DP-3-colorable if it has no cycles of length $\{4,9, a,b\}$, where $(a,b)\in \{(5,6), (5,7), (6,7), (6,8), (7,8)\}$.
\end{theorem}



In this paper, we use DP-coloring to improve the results in Theorem~\ref{MRW}.  To state our results, we have to introduce {\em extendability}. Let $G$ be a graph and $C$ be a subgraph of $G$. Then $(G,C)$ is DP-3-colorable if every DP-3-coloring of $C$ can be extended to $G$.

\begin{figure}[ht]\label{bad-9}
\includegraphics[scale=0.45]{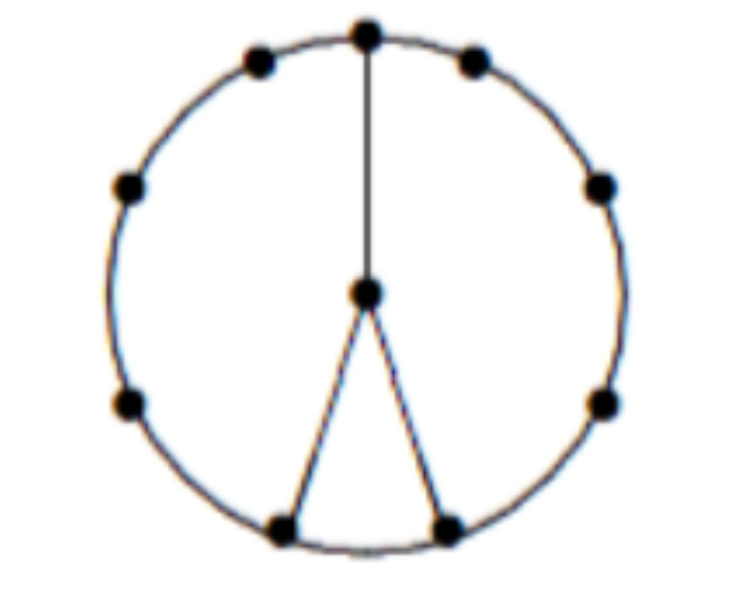}
\includegraphics[scale=0.45]{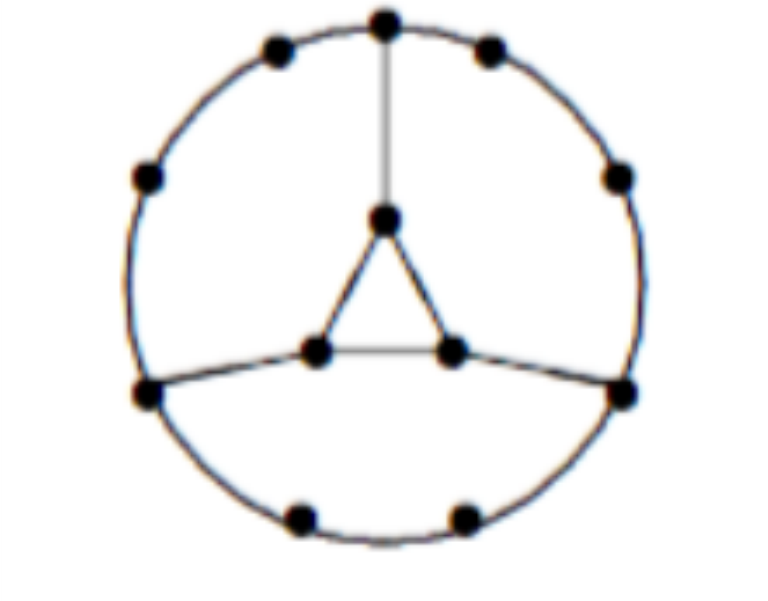}
\caption{bad 9-cycles.}
\end{figure}

A $9$-cycle $C$ is {\em bad} if it is the outer $9$-cycle in a subgraph isomorphic to the graphs in Figure~\ref{bad-9}. A $9$-cycle is {\em good} if it is not a bad $9$-cycle.

\begin{theorem}\label{main0}
Let $G$ be a planar graph that contains no $\{4,5\}$-cycles and $d^{\Delta}\ge 3$. Let $C_0$ be a $3$-, $6$-, $7$-, $8$-cycle or a good $9$-cycle in $G$. Then each DP-$3$-coloring of $C_0$ can be extended to $G$.
\end{theorem}

\begin{theorem}\label{main2}
Let $G$ be a planar graph that contains no $\{4,5,6\}$-cycles and $d^{\Delta}\ge 2$. Let $C_0$ be a cycle of length $7, 8, 9$ or $10$ in $G$. Then each DP-$3$-coloring of $C_0$ can be extended to $G$.
\end{theorem}

The proofs of Theorem~\ref{main0} and ~\ref{main2} use identification of vertices.  We shall note that  the planar graphs in the following corollary was not known to be 3-choosable.

\begin{corollary}
The following planar graphs are DP-$3$-colorable (thus also $3$-choosable):
\begin{itemize}
\item no $\{4,5\}$-cycles and $d^{\Delta}\ge 3$, or
\item no $\{4,5,6\}$-cycles and $d^{\Delta}\ge 2$.
\end{itemize}
\end{corollary}

\begin{proof}
Let $G$ be a planar graph under consideration. Note that $G$ is DP-3-colorable if $G$ contains no $3$-cycle. So we may assume that $G$ contains a $3$-cycle.  Then by Theorem~\ref{main0}, $G$ is DP-3-colorable when $d^{\Delta}\ge 3$. So we let $d^{\Delta}\ge 2$ and assume that $G$ contains no 
$\{4,5,6\}$-cycles.  By Theorem~\ref{DP3}, we may assume that $G$ contains a cycle of length in $\{7, 8, 9\}$. Now by Theorem~\ref{main2}, $G$ is DP-3-colorable.
\end{proof}

We use discharging method to prove the results.  One part of the proofs is to show some structures to be {\em reducible}, that is, a coloring outside of the structure can be extended to the whole graph.  The following lemma from \cite{LLYY18} provides a powerful tool to prove the reducibility.


\begin{lemma}\label{near-2-degenerate}\cite{LLYY18}
Let $k \ge 3$ and $H$ be a subgraph of $G$. If the vertices of $H$ can be ordered as $v_1, v_2, \ldots, v_{\ell}$ such that the following hold
\begin{itemize}
\item[(1)] $v_1v_{\ell}\in E(G)$,
and $v_1$ has no neighbor outside of $H$,
\item[(2)] $d(v_{\ell})\le k$ and $v_{\ell}$ has at least one neighbor in $G-H$,
\item[(3)] for each $2\le i\le \ell-1$, $v_i$ has at most $k-1$ neighbors in $G[\{v_1, \ldots, v_{i-1}\}]\cup (G-H)$,
\end{itemize}
then a DP-$k$-coloring of $G-H$ can be extended to a DP-$k$-coloring of $G$.
\end{lemma}

We end the introduction with some notations used in the paper. All graphs mentioned in this paper are simple.  A $k$-vertex ($k^+$-vertex, $k^-$-vertex) is a vertex of degree $k$ (at least $k$, at most $k$). The same notation will be applied to faces and cycles. We use $V(G)$ and $F(G)$ to denote the set of vertices and faces in $G$, respectively. An $(\ell_1, \ell_2)$-edge is an edge $e=v_1v_2$ with $d(v_i)=\ell_i$. An $(\ell_1, \ell_2, \ldots, \ell_k)$-face is a $k$-face $f=[v_1v_2\ldots v_k]$ with $d(v_i)=\ell_i$, respectively.  Recall that two faces are {\em adjacent} if they share a common edge, and are {\em intersecting} if they share a common vertex.  A vertex is {\em incident} to a face if it is on the face, and is {\em adjacent} to a face if it is not on the face but adjacent to a vertex on the face.  A vertex in $G$ is \emph{light} if it is incident to a $3$-face.  If $C$ is a cycle in an embedding of $G$, we use $int(C)$ and $ext(C)$ to denote the sets of vertices located inside and outside a cycle $C$, respectively. The cycle $C$ is called a {\em separating cycle} if $int(C)\ne\emptyset\ne ext(C)$.
An edge $uv\in E(G)$ is {\em straight}  if every $(u,c_1)(v,c_2)\in M_{uv}$ satisfies $c_1=c_2$. We note that if all edges in a subgraph are straight, then a DP-$3$-coloring on the subgraph is the same as a proper $3$-coloring.

\section{Proof of Theorem~\ref{main0}}

Let $(G,C_0)$ be a counterexample to Theorem ~\ref{main0} with minimum number of vertices, where $C_0$ is a $3$-,$6$-,$7$-,$8$-cycle or a good $9$-cycle. Below  we let $G$ be a plane graph. The following was shown in \cite{LLYY18} for every non-DP-3-colorable graphs.

\begin{lem}\label{minimum1}
For each $v\in G-C_0$, $d(v)\ge 3$ and for each $v\in C_0$, $d(v)\ge 2$.
\end{lem}

\begin{lem}\label{SC}
There exist no separating $\{3,6,7,8\}$-cycles or good $9$-cycle. 
\end{lem}

\begin{proof}
First of all, we note that $C_0$ cannot be a separating cycle. For otherwise, we may extend the coloring of $C_0$ to both inside $C_0$ and outside $C_0$, respectively, then combine them to get a coloring of $G$. So we may assume that $C_0$ is the outer face of the embedding of $G$.

Let $C\not=C_0$ be a separating $\{3,6,7,8\}$-cycle or good $9$-cycle in $G$.  By the minimality of $G$, the coloring of $C_0$ can be extended to $G-int(C)$. Now that $C$ is colored, thus by the minimality of $G$ again, the coloring of $C$ can be extended to $int(C)$. Combine inside and outside of $C$, we have a coloring of $G$, which is extended from the coloring of $C_0$, a contradiction.
\end{proof}

By Lemma~\ref{SC}, if $C$ is a bad $9$-cycle, then the subgraph in Figure~\ref{bad-9} that contains $C$ must be induced.  From now on, we will let $C_0$ be the outer face of $G$. Likewise, if $C_0$ contains a chord, then by Lemma~\ref{SC}, $G$ contains no other vertices, so the coloring on $C_0$ is also a coloring of $G$. Therefore, we may assume that $C_0$ is chordless as well.   A vertex is {\em internal} if it is not on $C_0$ and a face is {\em internal} if it contains no vertex of $C_0$.

For convenience, a $6^+$-face $f$ in $G$ is \emph{bad} if $d(f)=6$ and $f$ is adjacent to a $3$-face, otherwise, it is {\em good}. Let $f$ be a (3,3,3,3,3,3)-face adjacent to a $3$-face $f'$. We call the vertex $v$ on $f'$ but not on $f$ the {\em roof} of $f$, and $f$ the {\em base} of $v$.

\begin{lem}\label{big structure Lemma}
Let $f$ be an internal $6$-face in $G$ and $f_1$ be an internal $(3,3,4)$-face adjacent to $f$.  Then each of the followings holds:
\begin{enumerate}
\item[(a)] $f$ cannot contain vertices of another $3$-face;
\item[(b)] If $f$ is a $(3,3,3,3,3,4)$-face such that $f$ and $f_1$ share a common $(3,4)$-edge, then the other $(3,4)$-edge of $f_1$ cannot be on another internal $(3,3,3,3,3,4)$-face.
\item[(c)] If $f$ is a $(3,3,3,3,3,3)$-face, then $f_1$ cannot be adjacent to an internal $(3,3,3,3,3,4)$-face. This means a $4$-vertex on an internal $(3,3,3,3,3,4)$-face cannot be a roof.
\end{enumerate}
\end{lem}

\begin{proof}
(a) follows from the condition on the distance of triangles.  To show (b) and (c), let $f_1=[xyz]$ so that $xy$ is the common edge of $f_1$ and $f=[xyu_1u_2u_3u_4]$ and $d(x)\le d(y)$. Let $f_2=[zv_1v_2v_3v_4y]$ be the $(3,3,3,3,4)$-face adjacent to $f_1$.

(b) We have $d(y)=4$ and $d(x)=d(z)=d(u_i)=d(v_i)=3$ for $i\in [4]$. Order the vertices on $f$ and $f_2$ as
$$y,\ v_4,\ v_3,\ v_2,\ v_1,\ z,\ x,\ u_4,\ u_3,\ u_2,\ u_1.$$
Let $S$ be the set of vertices in the list. By Lemma ~\ref{near-2-degenerate}, a DP-$3$-coloring of $(G-S,C_0)$ can be extended to $(G,C_0)$, a contradiction.

(c) We have $d(z)=4$ and $d(x)=d(y)=d(u_i)=d(v_i)=3$ for $i\in [4]$, and $u_1=v_4$. Order the vertices on $f$ and $f_2$ as
$$x, \ z,\ v_1,\ v_2,\ v_3,\ y,\ u_1,\  u_2,\ u_3,\ u_4.$$
Let $S$ be the set of vertices in the list. By Lemma ~\ref{near-2-degenerate}, a DP-$3$-coloring of $(G-S,C_0)$ can be extended to $G$, a contradiction.
\end{proof}

\begin{lem}\label{bad-6}
Let $f=[v_1v_2v_3v_4v_5v_6]$ be an internal $6$-face that is adjacent to an internal $(3,3,3)$-face $f_1=[v_1v_2v_{12}]$, then $d(v_3)\geq4$ or $d(v_6)\geq4$.
\end{lem}

\begin{proof}
We assume that $d(v_3)=d(v_6)=3$, and use $v$ to denote the neighbor of $v_{12}$ other than $v_1,v_2$. First we may rename the lists of vertices in $\{v,v_{12},v_2,v_3,v_4\}$ so that each edge in $\{v_1v_2, vv_{12}, v_{12}v_2, v_2v_3, v_3v_4\}$ is straight.

Consider the graph $G'$ obtained from $G-\{v_{12},v_1, v_2, v_3, v_6\}$ by identifying $v_4$ and $v$.  We claim that no new cycles of length from $3$ to $5$ or multiple edges are created, for otherwise, there is a path of length $2, 3, 4$ or $5$ from $v$ to $v_4$ in $G-\{v_{12},v_1, v_2, v_3,v_6\}$, which together with $v_{12}, v_1, v_2, v_3$ forms a separating $\{6，7，8\}$-cycle or good $9$-cycle, a contradiction to Lemma~\ref{SC}. Clearly, $d^\Delta(G')\ge3$. Finally, we claim that no new chord in $C_0$ is formed in $G'$, for otherwise, $v\in C_0$ and $v_4$ is adjacent to a vertex on $C_0$, then there is a path between $v_4$ and $v$ on $C_0$ with length at most four, which with $v_3v_2v_{12}$ forms a separating $\{6,7,8\}$-cycle  or good $9$-cycle.

By minimality of $(G,C_0)$, the $DP$-$3$-coloring of $C_0$ can be extended to a $DP$-$3$-coloring $\phi$ of $G'$.  Now keep the colors of all vertices in $G'$ and color $v_4$ and $v$ with the color of the identified vertex. Now properly color $v_3$, and then color $v_{12}$ with the color of $v_3$, which we can do it because the edges $vv_{12}, v_{12}v_2, v_2v_3, v_3v_4$ are straight and the color of $v_3$ is different from the one on $v_4$ and $v$.  Now color $v_6, v_1, v_2$ properly in the order, we obtain a coloring of $G$, a contradiction. 
\end{proof}

\begin{lem}\label{2bad-6}
Let $P=xu_1u_2yv_1v_2z$ be a path in $int(C_0)$ and $f=[x'y'z']$ be an internal $(3,3,3)$-face so that $xx', yy', zz'\in E(G)$.  If $d(x)=d(u_1)=d(u_2)=3$, then $d(y)\ge 5$. (And similarly, if $d(z)=d(v_1)=d(v_2)=3$, then $d(y)\ge 5$.)
\end{lem}

\begin{proof}
Assume that $d(y)\le 4$. Since there is no separating $6$-cycles by Lemma~\ref{SC}, the $6$-cycles $xu_1u_2yy'x'$ and $yv_1v_2zz'y'$ are both $6$-faces. Then by Lemma~\ref{bad-6}, $d(y)=4$.  Let $y''$ be the fourth neighbor of $y$. We may rename the lists of vertices in $\{y, y', z'\}$ so that the edges $y''y, yy', y'z', z'z$ are straight.


Consider the graph $G'$ obtained from $G-\{x, u_1, u_2, y, y', x', z'\}$ by identifying $z$ and $y''$. Since $d^\Delta(G)\ge 3$, $v_1$ and $v_2$ cannot be on triangles. We claim that no new cycles of length from $1$ to $5$ are created, for otherwise, there is a path of length $2, 3, 4$ or $5$ from $y''$ to $z$ in $G-\{x, u_1, u_2, y, y', x', z'\}$, which together with $y, y', z'$ forms a separating $\{6,7,8\}$-cycle or good $9$-cycle, a contradiction to Lemma~\ref{SC}. Clearly, $d^\Delta(G')\ge3$. Finally, we claim that no new chord in $C_0$ is formed in $G'$, for otherwise, $y''\in C_0$ and $z$ is adjacent to a vertex on $C_0$, then there is a path between $y''$ and $z$ on $C_0$ with length at most four, which again forms a good separating cycle with $yy'x'$ of forbidden length.

By minimality of $(G,C_0)$, the $DP$-$3$-coloring of $C_0$ can be extended to a $DP$-$3$-coloring $\phi$ of $G'$. Now keep the colors of all vertices in $G'$ and color $y''$ and $z$ with the color of the identified vertex. For $u\in\{x, u_1, u_2, y, y', x', z'\}$, let $L^*(u)=L(u)\setminus\cup_{uv\in E(G)}\{c'\in L(u):(v,c)(u,c')\in C_{vu}$ and $(v,c)\in \phi\}$. Then $|L^*(z')|=|L^*(x)|=|L^*(u_2)|=|L^*(u_1)|\ge2$, $|L^*(y')|=|L^*(x')|=3$ and $|L^*(y)|\ge1$. So we can extend $\phi$ to a $DP$-$3$-coloring of $G$ by properly coloring $y$ and coloring $z'$ with the color of $y$, and coloring $u_2, u_1, x, x', y'$ in order, a contradiction.
\end{proof}

We use $\mu(x)$ to denote the initial charge of a vertex or face $x$ in $G$ and $\mu^*(x)$ to denote the final charge after the discharging procedure. We use $\mu(v) = 2d(v) - 6$ for each vertex $v$,  $\mu(f) = d(f) - 6$ for each face $f\not=C_0$, and $\mu(C_0)=d(C_0)+6$. Then by Euler formula,   $\sum_{x\in V(G)\cup F(G)} \mu(x)=0.$ To lead to a contradiction, we shall prove that $\mu^*(x)\ge 0$ for all $x\in V\cup F$ and $\mu^*(C_0)$ is positive.
For shortness, let $F_k=\{f: \text{ $f$ is a $k$-face and } V(f)\cap C_0\ne\emptyset\}$. 


We use the following discharging rules:

\begin{enumerate}[(R1)]
\item Each internal $4^+$-vertex gives $\frac{3}{2}$ to its incident $3$-face, and $\frac{1}{2}$ to its base or incident $(3,3,3,3,3,4)$-face.  Each internal $4$-vertex gives $1$ to its adjacent $(3,3,3)$-face and $\frac{1}{2}$ to its incident $6$-faces that are not adjacent to its adjacent $3$-face, and each internal $5^+$-vertex gives $2$ to its adjacent $(3,3,3)$-face and $\frac{1}{2}$ to its incident $6$-faces that are not adjacent to its adjacent $3$-face.

\item Each $7^+$-face or non-internal $6$-face other than $C_0$ gives $1$ to each of its adjacent internal $3$-faces and the rest to the outer face.  Each internal $6$-face gives $\frac{1}{2}$ to its adjacent internal $3$-face when it shares an $(3,4^+)$-edge with the $3$-face,  or contains a $4^+$-vertex that is not adjacent to a $(3,3,3)$-face, or it is a $(3,3,3,3,3,3)$-face.




\item The outer face $C_0$ gets $\mu(v)$ from each $v\in C_0$, gives $3$ to each intersecting $3$-face and $1$ to each adjacent bad $6$-face with an internal $3$-face.
\end{enumerate}

We first check the final charge of vertices in $G$. By (R3), each vertex on $C_0$ has final charge $0$.  So let $v$ be an internal vertex of $G$.  Then by Lemma~\ref{minimum1}, $d(v)\ge 3$.   Note the $\mu^*(v)=0$ if $d(v)=3$.

Let $d(v)=k\ge 5$. If $v$ is on a $3$-face, then it is not adjacent to other $3$-faces, so by (R1), it gives $\frac{3}{2}$ to the $3$-face, $\frac{1}{2}$ to each other incident face and possibly $\frac{1}{2}$ to its base (at most one by definition), so $\mu^*(v)\ge 2k-6-\frac{3}{2}-\frac{1}{2}\cdot k=\frac{3}{2}(k-5)\ge 0$.   If $v$ is adjacent to a $3$-face, then it is not on or adjacent to other $3$-faces, so by (R1), it gives at most $2$ to the $3$-face, and $\frac{1}{2}$ to each other incident $6$-faces that are not adjacent to the $3$-face, hence $\mu^*(f)\ge 2k-6-2-\frac{1}{2}(k-2)> 0$.  If $f$ is not on or adjacent to $3$-faces, then by (R1), its final charge is $\mu^*(f)\ge 2k-6-\frac{1}{2}k> 0$.

Let $d(v)=4$. Let $f_i$ for $1\le i\le4$ be the incident face of $v$ in clockwise order. First assume that $v$ is on a $3$-face.  By Lemma~\ref{big structure Lemma} (b) and (c), $v$ cannot be a roof and on a $(3,3,3,3,3,4)$-face at the same time, so by (R1), $v$ gives out at most $\frac{1}{2}$ to $6$-faces and $\frac{3}{2}$ to the $3$-face, thus $\mu^*(v)\ge 0$.  Now assume that $v$ is adjacent to a $3$-face.  Then $v$ cannot be adjacent other $3$-faces. By (R1), $v$ gives at most $1$ to the $3$-face and $\frac{1}{2}$ to each of the other $6$-faces that are not adjacent to the $3$-face, and $\mu^*(v)\ge 2-1-\frac{1}{2}\cdot 2=0$.  Finally assume that $v$ is not on or adjacent to any $3$-face.  Then by (R1), $\mu^*(v)\ge 2-\frac{1}{2}\cdot 4=0$.

Now we check the final charge of faces.  Let $d(f)=3$. If $f$ contains vertices of $C_0$, then by (R3), $\mu^*(f)=0$. So we assume that $f$ is internal.  If $f$ is incident with at least two $4^+$-vertices, then $f$ gets $\frac{3}{2}$ from each of the incident $4^+$-vertices by (R1), thus $\mu^*(f)\geq-3+\frac{3}{2}\cdot2=0.$ If $f$ is incident with exactly one $4^+$-vertex, then $f$ gets $\frac{3}{2}$ from the incident $4^+$-vertex by (R1) and gets $\frac{1}{2}$ from each of the incident $6^+$-face by (R2). 

Now we assume that $f=[x'y'z']$ is an internal $(3,3,3)$-face. Let $xx',yy',zz'\in E(G)$ and let $f_1, f_2, f_3$ be the three adjacent faces of $f$ so that $f_1$ contains $x,x',y',y$ and $f_2$ contains $y,y',z',z$. If $f$ is adjacent to three $7^+$- or non-internal $6$-faces, then it gets $1$ from each by (R2) and its final charge is at least $0$.  So we may assume that it is adjacent to an internal $6$-face, say $f_1$.  By Lemma~\ref{bad-6},  $f$ is adjacent to at least one internal $4^+$-vertex (say $y$) which is on $f_1$.  If $f$ is adjacent to three internal $6$-faces, then by Lemma~\ref{bad-6}, one of $x$ and $z$ is a $4^+$-vertex, and by Lemma~\ref{2bad-6}, either one of $x,y,z$ is a $5^+$-vertex, in which case by (R1), $\mu^*(f)\ge -3+2+1=0$, or they are all $4$-vertices, in which case by (R1), $\mu^*(f)\ge -3+1\cdot 3=0$, or one of them (say $x$) is a $3$-vertex and other two are $4$-vertices, in which case by Lemma~\ref{2bad-6}, $f_1$ and $f_3$ both contain $4^+$-vertices that are not adjacent to $f$ so by (R1) and (R2), $f$ gets $1+1$ from the two $4$-vertices and $\frac{1}{2}\cdot 2$ from $f_1$ and $f_3$.  Likewise, if $f_2$ and $f_3$ are both $7^+$- or non-internal $6$-faces, then by (R1) and (R2), $\mu^*(f)\ge -3+1+1\cdot 2=0$.   So we may assume that one of $f_2$ or $f_3$ is an internal $6$-face and the other is a $7^+$- or non-internal $6$-face.  If $f_3$ is an internal $6$-face, then by Lemma~\ref{bad-6}, $x$ or $z$ is a $4^+$-vertex, thus by (R1) $f$ gets $1\cdot 2$ from the two adjacent $4^+$-vertices and by (R2) $f$ gets $1$ from $f_2$. So we may assume that $f_2$ is an internal $6$-face and $f_3$ is a $7^+$- or non-internal $6$-face, and furthermore assume that $x,z$ are $3$-vertices and $d(y)=4$.  Now by Lemma~\ref{2bad-6},  $f_1$ and $f_2$ both contain $4^+$-vertices that are not adjacent to $f$, so by (R2), $f$ gets $\frac{1}{2}\cdot 2$ from $f_1$ and $f_2$, $1$ from $f_3$, and by (R1), $1$ from $y$, and we have $\mu^*(f)\ge -3+3=0$.

Since $G$ contains no $4$- or $5$-cycles, we only need to check the $6^+$-faces. If $d(f)\geq7$, then $f$ is adjacent to at most $\lfloor\frac{d(f)}{4}\rfloor$ $3$-faces, so after (R1), $\mu^*(f)\ge d(f)-6-\lfloor\frac{d(f)}{4}\rfloor\ge 0$.

Let $d(f)=6$. If $f$ is good or $f$ contains vertices of $C_0$, then $\mu^*(f)=0$. Now we assume that $f$ is an internal bad $6$-face that is adjacent to an internal $3$-face $f'=[xyz]$ on edge $xy$ with $d(x)\le d(y)$.

\begin{itemize}
\item If $d(x),d(y)\geq4$, then $f$ gives nothing to $f'$. So $\mu^*(f)=\mu(f)=0$.

\item If $d(x)=3$ and $d(y)\geq5$, then $f$ gets $\frac{1}{2}$ from $y$ and gives $\frac{1}{2}$ to $f'$. Thus $\mu^*(f)\geq6-6+\frac{1}{2}-\frac{1}{2}=0$.

\item If $d(x)=d(y)=3$, then by (R2), $f$ gives $\frac{1}{2}$ to $f'$ only when $f$ contains a $4^+$-vertex that is not adjacent to the $3$-face, in which case, $f$ gets $\frac{1}{2}$ from the $4^+$-vertex by (R1).  So we always have $\mu^*(f)\ge 0$.


\item Let $d(x)=3$ and $d(y)=4$. If $f$ is an internal $(3,3,3,3,3,4)$-face, then it gets $\frac{1}{2}$ from $y$, or else  $f$ contains another $4^+$-vertex, from which $f$ gets $\frac{1}{2}$. Thus $\mu^*(f)\geq6-6+\frac{1}{2}-\frac{1}{2}=0$.
\end{itemize}

We call a bad $6$-face $f$ in $F_6$ {\em special} if $f$ is adjacent to one internal $3$-face.

\begin{lem}
The final charge of $C_0$ is positive.
\end{lem}

\begin{proof}
Assume that $\mu^*(C_0)\le 0$.  Let $E(C_0, G-C_0)$ be the set of edges between $C_0$ and $G-C_0$. Let $e'$ be the number of edges in $E(C_0, G-C_0)$ that is not on a $3$-face and $x$ be the number of charges $C_0$ receives by (R3). Let $f_3=|F_3|$ and $f_6$ be the number of special $6$-faces. By (R3) and (R4), the final charge of $C_0$ is
\begin{align*}
\mu^*(C_0)&=d(C_0)+6+\sum_{v\in C_0} (2d(v)-6)-3f_3-f_6+x\\
&= d(C_0)+6+\sum_{v\in C_0} 2(d(v)-2)-2d(C_0)-3f_3-f_6+x\\
&=6-d(C_0)+2|E(C_0,G-C_0)|-3f_3-f_6+x\\
&\ge 6-d(C_0)+f_3+2e'-f_6+x,
\end{align*}
where the last equality follows from that each $3$-face in $F_3$ contains two edges in $E(C_0,G-C_0)$.

Note that for each special $6$-face $f$, no edge in $E(C_0,G-C_0)\cap E(f)$ is on $3$-faces. Then $e'\ge f_6$. When $e'=f_6$, $C_0$ is adjacent to at least three $6$-faces, so $e'=f_6\ge 3$, and it follows that $d(C_0)=9$ and $x=f_3=0$ and $e'=f_6=3$, in which case, we have a bad $9$-cycle as in the second graph in Figure~\ref{bad-9}. So we may assume that $e'\ge f_6+1$. Thus
$$\mu^*(C_0)\ge 6-d(C_0)+f_3+2e'-f_6+x\ge6-d(C_0)+f_3+x+f_6+2.$$

Since $\mu^*(C_0)\le 0$, $d(C_0)\ge 8$.  So if $f_6=1$, then $d(C_0)=9$ and $(f_3, x, e')=(0,0,2)$. Now that the $6$-face shares at most four vertices with $C_0$,  $C_0$ is adjacent to a $10^+$-face $f$ that contains at least five consecutive $2$-vertices on $C_0$, thus by (R3), $x\ge d(f)-6-\lceil\frac{d(f)-7}{4}\rceil>0$, a contradiction. 

Therefore, we may assume that $f_6=0$, and $f_3+2e'+x\le d(C_0)-6\le 3$. So $e'\le 1$.

Let $e'=1$. It follows that $f_3\le 1$.
\begin{itemize}
\item Let $f_3=1$.  Then $d(C_0)= 9$ and  $x=0$. Since $C_0$ is not a bad $9$-cycle, $C_0$ is adjacent to  a $7^+$-face $f$ and $f$ is adjacent to the $3$-face,  so by (R3), $f$ gives at least $1$ to $C_0$, that is, $x\geq1$, a contradiction.

\item Let $f_3=0$. Then $d(C_0)\ge 8$ and $x\le 1$. Note that $C_0$ is adjacent to a $9^+$-face $f$ that contains at least $d(C_0)-1$ consecutive $2$-vertices, thus by (R3), $f$ gives at least $d(f)-6-\lceil\frac{d(f)-(d(C_0)+1)}{4}\rceil\ge 2$ to $C_0$, a contradiction to $x\le 1$.
\end{itemize}

Finally let $e'=0$.  Then $f_3+x\le d(C_0)-6$, and each edge in $E(C_0,G-C_0)$ is on a $3$-face. Note that we may assume that $f_3>0$, for otherwise $G=C_0$.  Now follow the boundaries of the $7^+$-faces adjacent to $C_0$, each of the $f_3$ triangles is encountered twice, thus the $7^+$-faces do not give charge to at least $2f_3$ triangles, so $x\ge 2f_3$. It follows $f_3=1$ and $d(C_0)=9$.  In this case, $C_0$ is adjacent to a $10^+$-face $f$ that contains at least $7$ consecutive $2$-vertices on $C_0$. Then by (R3), $f$ gives at least $d(f)-6-\lceil\frac{d(f)-9}{4}\rceil\ge 3$ to $C_0$, a contradiction to $x=2$.
\end{proof}

\section{Proof of Theorem~\ref{main2}}



Let $(G,C_0)$ be a counterexample to Theorem ~\ref{main2} with minimum number of vertices, where $C_0$ is a $7$-, $8$-, $9$- or $10$-cycle.  Let $G$ be a plane graph.

\begin{lemma}\label{minimum}
For each $v\notin C_0$, $d(v)\ge 3$.
\end{lemma}
\begin{proof}
Let $v\notin C_0$ be a vertex with $d(v)\le2$. Any $\mathcal{M}_L$-coloring of $G-v$ can be extended to $G$ since $v$ has at most $d(v)$ elements of $L(v)$ forbidden by the colors selected for the neighbors of $v$, while $|L(v)|=3$.
\end{proof}


\begin{lemma}\label{separating}
The graph $G$ has no separating cycles of length $7, 8, 9$ or $10$.
\end{lemma}

\begin{proof}
First of all, we note that $C_0$ cannot be a separating cycle. For otherwise, we may extend the coloring of $C_0$ to both inside $C_0$ and outside $C_0$, respectively, then combine them to get a coloring of $G$. So we may assume that $C_0$ is the outer face of the embedding of $G$.

Let $C\not=C_0$ be a separating cycle of length $7, 8, 9$ or $10$ in $G$.  By the minimality of $G$, the coloring of $C_0$ can be extended to $G-int(C)$. Now that $C$ is colored, thus by the minimality of $G$ again, the coloring of $C$ can be extended to $int(C)$. Combine inside and outside of $C$, we have a coloring of $G$, which is extended from the coloring of $C_0$, a contradiction.
\end{proof}

So we may assume that $C_0$ is the outer face of the embedding of $G$ in the rest of this paper. Like in the previous section, we may assume that $C_0$ is chordless.  A face is {\em internal} if none of its vertices is on $C_0$, and a vertex is {\em internal} if it is not on $C_0$.

\begin{lemma}\label{7-face}
Let $f$ be an internal $7$-face that is adjacent to an internal $(3,3,3)$-face and is incident with at least six $3$-vertices.  Then none of the followings occur
\begin{itemize}
\item[(a)] $f$ contains a $(3,4)$-edge that is on an internal $(3,3, 4)$-face.
\item[(b)] $f$ contains seven $3$-vertices and is adjacent to an internal $(3,3,4^+)$-face.
\item[(c)] $f$ is adjacent to another internal $(3,3,3)$-face.
\end{itemize}
\end{lemma}

\begin{proof}
Let $f=[v_1v_2\cdots v_7]$, and $v_1v_2$ be the $(3,3)$-edge that is on an internal $(3,3,3)$-face $[v_1v_2v_{12}]$. Since $d^\Delta(G)\ge2$, by symmetry we may assume that $v_4v_5$ is on a $3$-face $[v_4v_5v_{45}]$.


(a) or (b): If $d(v_4)\le4$ and $d(v_5)=3$, then let $S$ be the set of vertices listed as:
$$v_2, v_3, v_4, v_{45}, v_5, v_6, v_7, v_1, v_{12}.$$
If $d(v_5)=4$, then let $S$ be the set of vertices listed as:
$$v_1, v_7, v_6, v_5, v_{45}, v_4, v_3, v_2, v_{12}.$$
By Lemma ~\ref{near-2-degenerate}, a DP-$3$-coloring of $G-S$ can be extended to $G$, a contradiction.

(c) Suppose otherwise that the $3$-face $[v_4v_5v_{45}]$ is an internal $(3,3,3)$-face. Let $v$ be the neighbor of $v_{45}$ not on $f$. Since $f$ is incident with at least six $3$-vertices, by symmetry we may assume that $d(v_6)=3$. We can rename the lists of vertices in $\{v,v_{45},v_4,v_5,v_6,v_7\}$ so that each edge in $\{v_7v_6, v_6v_5, v_5v_4, v_5v_{45}, v_{45}v\}$ is straight.

Consider the graph $G'$ obtained from $G-\{v_6, v_5, v_4, v_{45}\}$ by identifying $v_7$ and $v$. We claim that no new cycles of length from $3$ to $6$ are created, for otherwise, there is a path of length $3, 4, 5$ or $6$ from $v$ to $v_7$ in $G-\{v_6, v_5, v_4, v_{45}\}$, which together with $v_6,v_5,v_{45}$ forms a separating cycle of length $7,8,9$ or $10$, a contradiction to Lemma~\ref{separating}. Since none of $v_7$ and $v$ is on a triangle, $d^\Delta(G')\ge2$. Finally, we claim that no new chord in $C_0$ is formed in $G'$, for otherwise, $v\in C_0$ and $v_7$ is adjacent to a vertex on $C_0$, then there is a path between $v_7$ and $v$ on $C_0$ with length at most four, which again forms a separating cycle with $v_6v_5v_{45}$ of forbidden length.

By minimality of $(G,C_0)$, the $DP$-$3$-coloring of $C_0$ can be extended to a $DP$-$3$-coloring $\phi$ of $G'$.  Now keep the colors of all other vertices in $G'$ and color $v_7$ and $v$ with the color of the identifying vertex. For $x\in\{v_4,v_5,v_6,v_{45}\}$, let $L^*(x)=L(x)\setminus\cup_{ux\in E(G)}\{c'\in L(v):(u,c)(x,c')\in C_{ux}$ and $(u,c)\in \phi\}$. Then $|L^*(v_4)|=|L^*(v_{45})|\ge2$, $|L^*(v_5)|=3$ and $|L^*(v_6)|\ge1$. So we can extend $\phi$ to a $DP$-$3$-coloring of $G$ by color $v_6$ and $v_{45}$ with the same color and then color $v_4,v_5$ in order, a contradiction.
\end{proof}


We use $\mu(x)$ to denote the initial charge of a vertex or face $x$ in $G$ and $\mu^*(x)$ to denote the final charge after the discharging procedure. We use $\mu(v) = 2d(v) - 6$ for each vertex $v$,  $\mu(f) = d(f) - 6$ for each face $f\not=C_0$, and $\mu(C_0)=d(C_0)+6$. Then by Euler formula,   $\sum_{x\in V(G)\cup F(G)} \mu(x)=0.$ To lead to a contradiction, we shall prove that $\mu^*(x)\ge 0$ for all $x\in V\cup F$ and $\mu^*(C_0)$ is positive.

For shortness, let $F_k=\{f: \text{ $f$ is a $k$-face and } V(f)\cap C_0\ne\emptyset\}$.
We call a $7$-face $f$ in $F_7$ {\em special} if $f$ is adjacent to two internal $3$-faces. We call a $4$-vertex $v$ on a $7^+$-face $f$ {\em rich to $f$} if $v$ is not on a $3$-face adjacent to $f$.

\medskip

We have the following discharging rules:

\begin{enumerate}[(R1)]
\item Each internal $3$-face gets $\frac{3}{2}$ from each incident $4^+$-vertex and then gets its needed charge evenly from adjacent faces.
\item Each internal $7$-face gets $\frac{1}{2}$ from each incident rich $4$-vertex or $5^+$-vertex.
\item After (R1) and (R2), each $7^+$-face gives all its remaining charges to $C_0$.
\item The outer face $C_0$ gets $\mu(v)$ from each $v\in C_0$, gives $3$ to each face in $F_3$ and $1$ to each special $7$-face.
\end{enumerate}


\begin{lemma}\label{CHECK}
Every vertex $v$ and every face other than $C_0$ in $G$ has nonnegative final charge.
\end{lemma}
\begin{proof}
We first check the final charges of vertices in $G$. Let $v$ be a vertex in $G$. If $v\in C_0$, then by (R4) $\mu^*(v)=0$. If $v\notin C_0$, then by Lemma~\ref{minimum} $d(v)\ge3$. If $d(v)=3$, then $\mu^*(v)=\mu(v)=0$. Note that each vertex can be incident to at most one $3$-face since $d^\Delta(G)\geq 2$. Let $d(v)=4$. If $v$ is light, then $v$ gives $\frac{3}{2}$ to the incident $3$-face and $\frac{1}{2}$ to the incident $7$-face to which $v$ is rich by (R1) and (R2). If $v$ is not light, then $v$ gives at most $\frac{1}{2}$ to each incident face by (R2). In either case,  $\mu^*(v)\geq2\cdot4-6-\max\{\frac{3}{2}+\frac{1}{2}, \frac{1}{2}\cdot 4\}\geq0$.  If $d(v)\ge 5$, then $v$ gives $\frac{3}{2}$ to at most one  incident $3$-face and at most $\frac{1}{2}$ to each other incident face. So $\mu^*(v)\geq 2d(v)-6-\frac{3}{2}-\frac{1}{2}\cdot (d(v)-1)>0$.

Now we check the final charges of faces other than $C_0$ in $G$. Since $G$ contains no $4,5,6$-cycles, a $3$-face in $G$ is adjacent to three $7^+$-faces. Thus, by (R1) and (R4) each $3$-face has nonnegative final charge. Let $f$ be a $7^+$-face in $G$. By (R1) $f$ only needs to give $1$ to each adjacent internal $(3,3,3)$-face and $\frac{1}{2}$ to each adjacent internal $(3,3,4^+)$-face. Since $d^\Delta(G)\ge2$, $f$ is adjacent to at most $\lfloor \frac{d(f)}{3} \rfloor$ $3$-faces. If $d(f)\ge8$, then $\mu^*(f)\geq d(f)-6-1\cdot\lfloor\frac{d(f)}{3}\rfloor\ge 0$.
Let $d(f)=7$. Note that $f$ gives at most $1$ to each adjacent $3$-face by (R1). If $f$ is in $F_7$ or adjacent to at most one internal $3$-face, then by (R1) and (R4), $\mu^*(f)\geq7-6-\max\{1,1\cdot2-1\}=0$.

Therefore, we may assume that $f$ is an internal $7$-face and adjacent to two internal $3$-faces. If none of the $3$-faces is a $(3,3,3)$-face, or one of the two $3$-faces contains more than one $4^+$-vertex, then $f$ gives out at most $1$ to the $3$-faces, so $\mu^*(f)\ge 0$.  Thus, we may assume that $f$ is adjacent to a $(3,3,3)$-face $f_1$ and a $(3,3,3^+)$-face $f_2$. If $f_2$ shares a $(3,4^+)$-edge with $f$, then by Lemma~\ref{7-face} (a) $f$ contains a rich $4$-vertex or $5^+$-vertex, which gives $\frac{1}{2}$ to $f$ by (R2). So $\mu^*(f)\ge7-6-1-\frac{1}{2}+\frac{1}{2}=0$.
If $f_2$ shares a $(3,3)$-edge with $f$, then by Lemma~\ref{7-face} (b) and (c), $f$ contains at least one $4^+$-vertex if $f_2$ is a $(3,3,4^+)$-face and at least two $4^+$-vertices if $f_2$ is a $(3,3,3)$-face, respectively. By (R2) $f$ gets $\frac{1}{2}$ from each incident rich $4$-vertex or $5^+$-vertex. So $\mu^*(f)\ge7-6-\max\{1+\frac{1}{2}-\frac{1}{2}, 1\cdot2-\frac{1}{2}\cdot2\}=0$.
\end{proof}

\begin{lemma}
The final charge of $C_0$ is positive.
\end{lemma}

\begin{proof}
Let $E(C_0, G-C_0)$ be the set of edges between $C_0$ and $G-C_0$. Let $e'$ be the number of edges in $E(C_0, G-C_0)$ that is not on a $3$-face and $x$ be the number of charges $C_0$ receives by (R3). Let $f_3=|F_3|$ and $f_7$ be the number of special $7$-faces. By (R3) and (R4), the final charge of $C_0$ is at least
\begin{align*}
\mu^*(C_0)&=d(C_0)+6+\sum_{v\in C_0} (2d(v)-6)-3f_3-f_7+x\\
&\ge d(C_0)+6+\sum_{v\in C_0} 2(d(v)-2)-2d(C_0)-3f_3-f_7+x\\
&\ge 6-d(C_0)+2|E(C_0,G-C_0)|-3f_3-f_7+x\\
&=6-d(C_0)+f_3+2e'-f_7+x,
\end{align*}
where the last equality follows from that each $3$-face in $F_3$ contains two edges in $E(C_0,G-C_0)$.

Let $f$ be a $7^+$-face adjacent to $C_0$.  A path on $f$ is {\em charge-friendly} if no vertex on it is on a triangle that needs charge from $f$ (which means triangles on the paths are in $F_3$). Let $P$ be a charge-friendly path on $f$. Then $f$ gives at least $d(f)-6-\lfloor\frac{d(f)+1-|V(P)|}{3}\rfloor$ to $C_0$, and thus
\begin{equation}\label{eq-4}
x\ge d(f)-6-\left\lfloor\frac{d(f)+1-|V(P)|}{3}\right\rfloor\ge \frac{2}{3}(d(f)-9)+\frac{|V(P)|-1}{3}.
\end{equation}

Since $d^\Delta(G)\ge2$, each special $7$-face must share exactly one edge or one vertex with $C_0$ and each edge in $E(C_0,G-C_0)\cap E(f)$ is not on $3$-faces. Then $e'\ge f_7$, with equality only if $e'=f_7=d(C_0)$ and $f_3=0$, in which case, $\mu^*(C_0)\ge6-d(C_0)+d(C_0)>0$. So we may assume that $e'\ge f_7+1$. Then
\begin{quote}
$f_7=0$ when $d(C_0)\le 8$, $f_7\le 1$ when $d(C_0)=9$, and $f_7\le 2$ when $d(C_0)=10$,
\end{quote}
for otherwise, $\mu^*(C_0)\ge 6-d(C_0)+f_3+2e'-f_7+x\ge6-d(C_0)+f_3+x+f_7+2>0$. Now assume that $\mu^*(C_0)\le 0$. We consider a few cases.

{\bf Case 1.} $f_7=0$.  From $\mu^*(C_0)\ge 6-d(C_0)+f_3+2e'-f_7+x=6-d(C_0)+f_3+2e'+x$, we have $f_3+2e'+x\le d(C_0)-6\le 10-6=4$. So $e'\le 2$.

Let $e'=2$. Then $d(C_0)=10$ and $f_3=x=0$. It follows that $G$ is adjacent to a $7^+$-face $f$ that contains at least four consecutive $2$-vertices, thus $f$ contains a charge-friendly path $P$ with $|V(P)|\ge6$, so by \eqref{eq-4} $x\ge \frac{2}{3}(7-9)+\frac{6-1}{3}>0$, a contradiction.

Let $e'=1$. It follows that $f_3\le 2$.
\begin{itemize}
\item If $f_3=2$, then $d(C_0)=10$ and $x=0$. Now $C_0$ is adjacent to a $7^+$-face that contains a path with a triangle at one end and having at least two consecutive $2$-vertices, thus, $f$ contains a charge-friendly path $P$ with $|V(P)|\ge6$, so by \eqref{eq-4} $x\ge \frac{2}{3}(7-9)+\frac{6-1}{3}>0$, a contradiction.

\item If $f_3=1$, then $d(C_0)\ge 9$. Note that $C_0$ contains at most three $3^+$-vertices. If $d(C_0)=9$, then $x=0$ and $C_0$ is adjacent to a $7^+$-face that contains a path with a triangle at one end and having at least three consecutive $2$-vertices. Thus, $f$ contains a charge-friendly path $P$ with $|V(P)|\ge7$, so by \eqref{eq-4} $x\ge \frac{2}{3}(7-9)+\frac{7-1}{3}>0$, a contradiction. If $d(C_0)=10$, then $x\le1$ and $C_0$ is adjacent to a $8^+$-face that contains a path with a triangle at one end and having at least four consecutive $2$-vertices. Thus, $f$ contains a charge-friendly path $P$ with $|V(P)|\ge8$, so by \eqref{eq-4} $x\ge \frac{2}{3}(8-9)+\frac{8-1}{3}>1$, a contradiction.


\item If $f_3=0$, then $d(C_0)\ge 8$ and $x\le2$. Note that $C_0$ is adjacent to a $9^+$-face $f$ that contains at least $d(C_0)-1$ consecutive $2$-vertices, thus $f$ contains a charge-friendly path of at least $d(C_0)+1$ vertices, so $x\ge \frac{2}{3}(9-9)+\frac{d(C_0)+1-1}{3}>2$ by \eqref{eq-4}, a contradiction.
\end{itemize}

Finally let $e'=0$.  Then $f_3+x\le d(C_0)-6$, and each edge in $E(C_0,G-C_0)$ is on a $3$-face. Note that we may assume that $f_3>0$, for otherwise $G=C_0$.  Now follow the boundaries of the $7^+$-faces adjacent to $C_0$, each of the $f_3$ triangles is encountered twice, thus the $7^+$-faces do not give charge to at least $2f_3$ triangles.  
So $x\ge 2f_3$. It follows $f_3=1$ and $d(C_0)\ge 9$.  In this case, $C_0$ is adjacent to a $9^+$-face $f$ that contains at least $d(C_0)-2$ consecutive $2$-vertices and a triangle at one end, thus $f$ contains a charge-friendly path of at least $d(C_0)+2$ vertices, so So by \eqref{eq-4}, $x\ge \frac{2}{3}(9-9)+\frac{d(C_0)+2-1}{2}>3$, a contradiction.


{\bf Case 2.} $f_7=1$. As $\mu^*(C_0)\ge 6-d(C_0)+f_3+2e'-f_7+x\ge 6-d(C_0)+f_3+x+f_7+2$, either $d(C_0)=9$ and $(f_3, x, e')=(0,0,2)$, or $d(C_0)=10$ and $f_3+x+2e'\le 5$. In the former case, $C_0$ is adjacent to a $9^+$-face that contains seven $2$-vertices, thus by \eqref{eq-4}, $x\ge\frac{2}{3}(9-9)+\frac{7-1}{3}>0$, a contradiction. Consider the latter case. It follows that $e'=2$ and $f_3+x\le 1$. So if $f_3=0$, then $C_0$ is adjacent to a $9^+$-face that contains eight consecutive $2$-vertices, thus $x\ge\frac{2}{3}(9-9)+\frac{8-1}{3}>2$ by \eqref{eq-4}, a contradiction; if $f_3=1$, then $C_0$ is adjacent to a $7^+$-face $f$ that contains at least three consecutive $2$-vertices and a triangle at one end, thus $f$ contains a charge-friendly path of at least $6$ vertices, so thus $x\ge\frac{2}{3}(7-9)+\frac{6-1}{3} >0$ by \eqref{eq-4}, a contradiction again.

{\bf Case 3.} $f_7=2$. Then $\mu^*(C_0)\ge 6-10+f_3+2e'-f_7+x\ge -4+f_3+x+f_7+2$, we have $f_3=x=0$ and $e'=3$. Thus, $C_0$ is a $10$-face and the two $7$-faces must share an edge in $E_0$. Then $C_0$ is adjacent to a $8^+$-face $f'$ that contains seven consecutive $2$-vertices. Thus $x\ge \frac{2}{3}(8-9)+\frac{7-1}{3}>0$ by \eqref{eq-4}, a contradiction.
\end{proof}

{\bf Acknowledgement:} The authors would like to thank Runrun Liu for her careful reading and many valuable comments.


\begin{thebibliography}{99}

\bibitem{A00}
Noga Alon, Degrees and choice numbers, {\em Random Structures $\&$ Algorithms}, 16(2000) 364--368.

\bibitem{B16}
Anton Bernshteyn. The asymptotic behavior of the correspondence chromatic number, {\em Discrete Math.}, 339(2016) 2680--2692.

\bibitem{B17}
Anton Bernshteyn, The Johansson--Molloy Theorem for DP-Coloring, \emph{arXiv:1708.03843}.

\bibitem{BK17a}
Anton Bernshteyn, Alexandr Kostochka. On differences between DP-coloring and list coloring, \emph{arXiv:1705.04883}.

\bibitem{BK17b}
Anton Bernshteyn, Alexandr Kostochka, Sharp Dirac's Theorem for DP-Critical Graphs, \emph{arXiv:1609.09122}.

\bibitem{BKP17}
Anton Bernshteyn, Alexandr Kostochka, S. Pron, On DP-coloring of graphs and multigraphs, \emph{Siberian Mathematical Journal}, {\bf 58} (2017), 28--36

\bibitem{BKZ17}
Anton Bernshteyn, Alexandr Kostochka, Xuding Zhu, DP-colorings of graphs with high chromatic number, \emph{arXiv:1703.02174}.


\bibitem{B96}
Oleg Borodin, Structural properties of plane graphs without adjacent triangles and an application to 3-colorings. {\em J. Graph Theory} 21 (1996), no. 2, 183-–186.

\bibitem{B13}
Oleg Borodin, Colorings of plane graphs: A survey, {\em Disc. Math.}, 313 (2013), pp. 517–-539.

\bibitem{BG10}
O.V. Borodin, A. Glebov, \emph{Planar Graphs with Neither 5-Cycles Nor Close 3-Cycles Are 3-Colorable}, J. Graph Theory, (2010), 1-31.


\bibitem{BGRS05}
Oleg Borodin, Aleksey Glebov, Andrea Raspaud, and Mohammad Salavatipour, Planar graphs without cycles of length from 4 to 7 are 3-colorable, {\em J. Combin. Theory, Ser. B},
93 (2005), 303--311.

\bibitem{BR03}
Oleg Borodin, Andre Raspaud, \emph{A sufficient condition for planar graphs to be 3-colorable}, J. Combinatorial Theory, Series B, (2003), 17--27.

\bibitem{DP17}
Zden\v{e}k Dvo\v{r}\'{a}k, Luke Postle, Correspondence coloring and its application to list-coloring planar graphs without cycles of length $4$ to $8$, {\em J. Combin. Theory Ser. B}, 129 (2018), 38–54.

\bibitem{DKT16}
Zden\v{e}k Dvo\v{r}\'{a}k, Daniel Kral, Robin Thomas, \emph{Three-coloring triangle-free graphs on surfaces V. Coloring planar graphs with distant anomalies,} arXiv:0911.0885.

\bibitem{ERT79}
Paul Erd\H{o}s, Arthur Rubin, Herbert Taylor, Choosability in graphs, {\em Congr. Numer.}, 26(1979) 125--157.

\bibitem{H69}
I. Havel, \emph{On a Conjecture of B. Gr\"{u}nbaum}, J. Combinatorial Theory, (1969), 184-186.

\bibitem{G59}
Herbert Gr\"otzsch, Ein Dreifarbensatz f\"ur Dreikreisfreie Netze auf der Kugel, {\em Math.-Natur. Reihe} 8(1959) 109-120.


\bibitem{KO17a}
Soeg-Jin Kim, Kenta Ozeki, A note on a Brooks' type theorem for DP-coloring, \emph{arXiv:1709.09807v1}.

\bibitem{KO17b}
Seog-Jin Kim, Kenta Ozeki, A Sufficient condition for DP-$4$-colorability, \emph{arXiv:1709.09809v1}.

\bibitem{KY17}
Seog-Jin Kim, Xiaowei Yu, Planar graphs without $4$-cycles adjacent to triangles are DP-$4$-colorable, \emph{arXiv:1712.08999}



\bibitem{LLYY18}
Runrun Liu, Sarah Loeb, Yuxue Yin, Gexin Yu, DP-$3$-coloring of some planar graphs, \emph{arXiv:1802.09312}.

\bibitem{LLRYY18b}
Runrun Liu, Sarah Loeb, Martin Rolek,  Yuxue Yin, Gexin Yu, DP-$3$-coloring of some more planar graphs, in preparation.


\bibitem{LLNSY18}
Runrun Liu, Xiangwen Li, Kittikorn Nakprasit, Pongpat Sittitrai, Gexin Yu, DP-$4$-colorability of planar graphs without given two adjacent cycles, submitted.

\bibitem{MRW06}
Mickael Montassier, Andre Raspaud, Weifan Wang. Bordeaux $3$-color conjecture and $3$-choosability, {\em Discrete Mathematics} 306(2006) 573--579.

\bibitem{T94}
Cartessen Thomassen, Every planar graph is $5$-choosable, {\em J. Combin. Theory Ser. B} 62(1994) 180--181.

\bibitem{T95}
Cartessen Thomassen, $3$-list-coloring planar graphs of girth $5$, {\em J. Combin. Theory Ser. B} 64(1995) 101--107.

\bibitem{V76}
Vadim Vizing, Vertex colorings with given colors, {\em Metody Diskret. Analiz, Novosibirsk}, 29(1976) 3-10(in Russian).

\bibitem{V93}
Margit Voigt, List coloring of planar graphs, {\em Discrete Math.,} 120(1993) 215--219.

\bibitem{V95}
Margit Voigt, A not $3$-choosable planar graph without $3$-cycles, {\em Discrete Math.,} 146(1995) 325--328.



\end{thebibliography}
\end{document}